\documentclass[10pt; a4paper]{amsart} 
\usepackage{amscd, amsfonts, amsmath, amssymb, amsthm, arydshln, fixmath, graphicx, mathrsfs, overpic}
\usepackage{hyperref}
\usepackage[all]{xy} 
\usepackage{tikz}
\usepackage{pgfplots}
\usepackage{pgflibraryarrows}
\usepackage{pgflibrarysnakes}
\makeindex

\theoremstyle{definition} 
\newtheorem{Unity}{Unity}[section] 
\newtheorem*{Definition*}{Definition} 
\newtheorem{Definition}[Unity]{Definition} 

\theoremstyle{plain} 
\newtheorem*{Theorem*}{Theorem}
\newtheorem{Theorem}[Unity]{Theorem}
\newtheorem{Proposition}[Unity]{Proposition}
\newtheorem{Corollary}[Unity]{Corollary}
\newtheorem{Lemma}[Unity]{Lemma}

\theoremstyle{remark} 
\newtheorem*{Remark*}{Remark}

\numberwithin{Unity}{section}

\newcommand{\PP}{\mathbb{P}}
\newcommand{\E}{\mathscr{E}}
\newcommand{\Ec}{\mathcal{E}}

\newcommand{\F}{\mathscr{F}}
\newcommand{\G}{\mathscr{G}}
\newcommand{\Ls}{\mathscr{L}}
\newcommand{\M}{\mathfrak{M}}

\newcommand{\Ox}{\mathscr{O}}
\newcommand{\Ps}{\mathscr{P}}

\newcommand{\rk}{\mathrm{rk}}

\newcommand{\HNP}{\mathrm{HNP}}

\newcommand{\Pic}{\mathrm{Pic}}
\newcommand{\Frob}{\mathrm{Frob}}
\newcommand{\Quot}{\mathrm{Quot}}

\newcommand{\Supp}{\mathrm{Supp}}
\newcommand{\Omg}{\mathrm{\Omega}}

\newcommand{\ConPgn}{\mathfrak{ConPgn}}

\begin{document}

\title{Frobenius Stratification of Moduli Spaces of Vector Bundles in Positive characteristic. II}
\author{Lingguang Li}
\address{School of Mathematical Sciences, Tongji University, Shanghai, P. R. China}
\email{LiLg@tongji.edu.cn}
\date{March 11, 2018}
\thanks{Partially supported by National Natural Science Foundation of China (Grant No. 11501418) and Shanghai Sailing Program.}

\begin{abstract} Let $X$ be a smooth projective curve of genus $g(X)\geq 1$ over an algebraically closed field $k$ of characteristic $p>0$, $\M^s_X(r,d)$ the moduli space of stable vector bundles of rank $r$ and degree $d$ on $X$. We study the Frobenius stratification of $\M^s_X(r,d)$ in terms of Harder-Narasimhan polygons of Frobenius pull backs of stable vector bundles and obtain the irreducibility and dimension of each non-empty Frobenius stratum in case $(p,g,r)=(3,2,3)$.
\end{abstract}
\maketitle
$$To~The~Memory~of~Professor~Michel~Raynaud$$

\section{Introduction}

Let $k$ be an algebraically closed field of characteristic $p>0$, $X$ a smooth projective curve of genus $g$ over $k$. The absolute Frobenius morphism $F_X:X\rightarrow X$ is induced by $\Ox_X\rightarrow \Ox_X$, $f\mapsto f^p$. Let $\M^s_X(r,d)$ be the moduli space of stable vector bundles of rank $r$ and degree $d$ on $X$.

For any vector bundle $\E$ on $X$, by the Harder-Narasimhan filtration of $\E$, we can define the \emph{Harder-Narasimhan Polygon} $\HNP(\E)$, which is a convex polygon in the coordinate plane of $\mathrm{rank}$-$\mathrm{degree}$ (cf. \cite[Section 3]{Shatz77}).

Fix integers $m$ and $n$ with $m>0$. Let $\ConPgn(m,n)$ be the set consists of all convex polygons in the coordinate plane which start at original point end at point $(m,n)$ and their vertexes are integral lattice points. Then there is a natural partial order structure, denoted by $\succcurlyeq$, on the set $\ConPgn(m,n)$.

In general, the semistability of vector bundles does not preserved under Frobenius pull back $F^*_X$ (cf. \cite{Gieseker73}, \cite{Raynaud82}). Thus, there is a natural set-theoretic map
\begin{eqnarray*}
S^s_{\Frob}:\M^s_X(r,d)&\rightarrow&\ConPgn(r,pd)\\
{[\E]}&\mapsto&\HNP(F^*_X(\E))
\end{eqnarray*}
For any $\Ps\in\ConPgn(r,pd)$, denotes
\begin{eqnarray*}
S_X(r,d,\Ps)&:=&\{[\E]\in\M^s_X(r,d)~|~\HNP(F^*_X(\E))=\Ps\}.\\
S_X(r,d,\Ps^+)&:=&\{[\E]\in\M^s_X(r,d)~|~\HNP(F^*_X(\E))\succcurlyeq\Ps\}.
\end{eqnarray*}
Then we have a canonical stratification of $\M^s_X(r,d)$ in terms of Harder-Narasimhan polygons of Frobenius pull backs of stable vector bundles, we call this the \emph{Frobenius stratification}.
By \cite[Theorem 3]{Shatz77}, S. S. Shatz showed that the Frobenius stratum $S_X(r,d,\Ps^+)$ is a closed subvariety of $\M^s_X(r,d)$ for any $\Ps\in\ConPgn(r,pd)$.

Our main themes is to study the geometric properties of Frobenius strata, such as non-emptiness, irreducibility, connectedness, smoothness, dimension and so on. However, very little is known about these Frobenius strata. Some results are only known in special cases for small values of $p$, $g$, $r$ and $d$. Joshi-Ramanan-Xia-Yu \cite{JRXY06} have given a complete description of moduli space $\M^s_{X}(2,d)$ when $p=2$ and $g\geq2$. In \cite{Li18} the author provided a complete classification of Frobenius strata in $\M^s_{X}(3,0)$ when $p=3$ and $g=2$. In the higher rank case, the author have obtained the geometric properties of a special Frobenius stratum in $\M^s_{X}(r,d)$ when $p|r$ and $g\geq2$ in \cite{Li14} and \cite{Li18}. Anther results about Frobenius stratification can be found in \cite{D09}\cite{JX00}\cite{LanP08}\cite{LasP02}\cite{LasP04}\cite{Oss06}\cite{Zh17} for $r=2$ and \cite{Li14}\cite{LiuZhou13}\cite{MehtaPauly07}\cite{Zh17} for $r>2$.

In this paper, we mainly study the Frobenius stratification of $\M^s_X(3,d)$ for any degree $d$, where $X$ is a smooth projective curve of genus $2$ over an algebraically closed field $k$ of characteristic $3$. In this case, there are $4$ possible Harder-Narasimhan polygons $\{\Ps_1(d),\Ps_2(d),\Ps_3(d),\Ps_4(d)\}$ for Frobenius destabilised stable vector bundles of rank $3$ and degree $d$ (See Section 2). The main result of this paper is the following Theorem \ref{Thm:FrobStra}.

\begin{Theorem}$($Theorem \ref{Thm:FrobStra}$)$
Let $k$ be an algebraically closed field of characteristic $3$, and $X$ a smooth projective curve of genus $2$ over $k$. Then $\overline{S_X(3,d,\Ps_i(d))}=S_X(3,d,\Ps^+_i(d))$, and $S_X(3,d,\Ps_i(d))$ $($resp. $S_X(3,d,\Ps^+_i(d))$$)$ are irreducible quasi-projective $($resp. projective$)$ varieties for $1\leq i\leq 4$, $$\dim S_X(3,d,\Ps_i(d))=\dim S_X(3,d,\Ps^+_i(d))=
\begin{cases}
5, i=1\\
5, i=2\\
4, i=3\\
2, i=4\\
\end{cases}$$
\end{Theorem}

\section{Classification of Frobenius Harder-Narasimhan Polygons}

In this section, we will determine all of the possible Harder-Narasimhan polygons of $F^*_X(\E)$ for any Frobenius destabilized stable vector bundles $[\E]\in\M^s_X(3,d)$, where $X$ is a smooth projective curve of genus $2$ over an algebraically closed field $k$ of characteristic $3$.

\begin{Lemma}[N. I. Shepherd-Barron \cite{Shepherd-Barron98} and V. Mehta, C. Pauly \cite{MehtaPauly07}]\label{Lemma:InsHN}
Let k be an algebraically closed field of characteristic $p>0$, $X$ a smooth projective curve of genus $g\geq 2$ over $k$, $\E$ a semistable vector bundle on $X$. Let $0=\E_0\subset\E_1\subset\cdots\subset\E_{m-1}\subset\E_m=F^*_X(\E)$ be the Harder-Narasimhan filtration of $F^*_X(\E)$. Then $\mu(\E_i/\E_{i-1})-\mu(\E_{i+1}/\E_i)\leq 2g-2$, for any $1\leq i\leq m-1$.
\end{Lemma}

According to the Lemma \ref{Lemma:InsHN}, there are $4$ possible Harder-Narasimhan polygons $$\{\Ps_1(d),\Ps_2(d),\Ps_3(d),\Ps_4(d)\}$$ for Frobenius destabilised stable vector bundles of rank $3$ and degree $d$ in the case $(p,g)=(3,2)$.

\begin{center}
\begin{tabular}{cccc}
\begin{picture}(80,110)
\linethickness{0.5pt}
\put(0,0){\vector(1,0){80}} 
\put(0,0){\vector(0,1){105}} 
\multiput(0,0)(20,0){4}{\line(0,1){2}}
\put(18,-8){1}
\put(38,-8){2}
\put(58,-8){3}
\put(-20,37){$d$+1}
\put(-16,87){3$d$}
\put(65,5){\tiny{rank}}
\put(2,100){\tiny{deg}}
\thinlines

\multiput(0,90)(4,0){15}{\line(1,0){2}} 
\multiput(0,40)(4,0){5}{\line(1,0){2}} 
\multiput(20,0)(0,4){10}{\line(0,1){2}} 
\multiput(60,0)(0,4){23}{\line(0,1){2}} 
\put(0,0){\line(1,2){20}} 
\put(20,40){\line(4,5){40}}
\put(50,92){\tiny{$\Ps_1(d)$}}
\end{picture}
&
\begin{picture}(80,110)
\linethickness{0.5pt}
\put(0,0){\vector(1,0){80}} 
\put(0,0){\vector(0,1){105}} 
\multiput(0,0)(20,0){4}{\line(0,1){2}}
\put(18,-8){1}
\put(38,-8){2}
\put(58,-8){3}
\put(-24,67){2$d$+1}
\put(-16,87){3$d$}
\put(65,5){\tiny{rank}}
\put(2,100){\tiny{deg}}
\thinlines

\multiput(0,90)(4,0){15}{\line(1,0){2}} 
\multiput(0,70)(4,0){10}{\line(1,0){2}} 
\multiput(40,0)(0,4){18}{\line(0,1){2}} 
\multiput(60,0)(0,4){23}{\line(0,1){2}} 
\put(40,70){\line(1,1){20}}
\put(50,92){\tiny{$\Ps_2(d)$}}
\begin{tikzpicture}
\draw (0,0)--(1.4,2.45);
\end{tikzpicture}
\end{picture}
&
\begin{picture}(80,110)
\linethickness{0.5pt}
\put(0,0){\vector(1,0){80}} 
\put(0,0){\vector(0,1){105}} 
\multiput(0,0)(20,0){4}{\line(0,1){2}}
\put(18,-8){1}
\put(38,-8){2}
\put(58,-8){3}
\put(-20,37){$d$+1}
\put(-24,67){2$d$+1}
\put(-16,87){3$d$}
\put(65,5){\tiny{rank}}
\put(2,100){\tiny{deg}}
\thinlines

\multiput(0,90)(4,0){15}{\line(1,0){2}} 
\multiput(0,70)(4,0){10}{\line(1,0){2}} 
\multiput(0,40)(4,0){5}{\line(1,0){2}} 
\multiput(20,0)(0,4){10}{\line(0,1){2}} 
\multiput(40,0)(0,4){18}{\line(0,1){2}} 
\multiput(60,0)(0,4){23}{\line(0,1){2}} 
\put(0,0){\line(1,2){20}} 
\put(20,40){\line(2,3){20}}
\put(40,70){\line(1,1){20}}
\put(50,92){\tiny{$\Ps_3(d)$}}
\end{picture}
&
\begin{picture}(80,110)
\linethickness{0.5pt}
\put(0,0){\vector(1,0){80}} 
\put(0,0){\vector(0,1){105}} 
\multiput(0,0)(20,0){4}{\line(0,1){2}}
\put(18,-8){1}
\put(38,-8){2}
\put(58,-8){3}
\multiput(0,0)(0,10){10}{\line(1,0){2}}
\put(-20,47){$d$+2}
\put(-24,77){2$d$+2}
\put(-16,87){3$d$}
\put(65,5){\tiny{rank}}
\put(2,100){\tiny{deg}}
\thinlines

\multiput(0,90)(4,0){15}{\line(1,0){2}} 
\multiput(0,80)(4,0){10}{\line(1,0){2}} 
\multiput(0,50)(4,0){5}{\line(1,0){2}} 
\multiput(20,0)(0,4){13}{\line(0,1){2}} 
\multiput(40,0)(0,4){20}{\line(0,1){2}} 
\multiput(60,0)(0,4){23}{\line(0,1){2}} 
\put(0,0){\line(2,5){20}} 
\put(20,50){\line(2,3){20}}
\put(40,80){\line(2,1){20}}
\put(50,92){\tiny{$\Ps_4(d)$}}
\end{picture}
\end{tabular}
\end{center}

\section{Construction of Stable Vector Bundles}

\begin{Definition}(\cite{JRXY06}\cite{Sun08i})\label{CanFil}
Let $k$ be an algebraically closed field of characteristic $p>0$, $X$ a smooth projective curve over $k$. For
any coherent sheaf $\F$ on $X$, let
$$\nabla_{\mathrm{can}}:F^*_X{F_X}_*(\F)\rightarrow
F^*_X{F_X}_*(\F)\otimes_{\Ox_X}\Omg^1_X$$
be the canonical connection on $F^*_X{F_X}_*(\F)$.
Set
\begin{eqnarray*}
V_1&:=&\ker(F^*_X{F_X}_*(\F)\twoheadrightarrow\F),\\
V_{l+1}&:=&\ker\{V_l\stackrel{\nabla_{\mathrm{can}}}{\rightarrow}F^*_X{F_X}_*(\F)
\otimes_{\Ox_X}\Omg^1_X\rightarrow (F^*_X{F_X}_*(\F)/V_l)
\otimes_{\Ox_X}\Omg^1_X\}
\end{eqnarray*}
The filtration ${\mathbb{F}^{\mathrm{can}}_\F}_\bullet:F^*_X{F_X}_*(\F)=V_0\supset V_1\supset V_2\supset\cdots$
is called the \emph{canonical filtration} of $F^*_X{F_X}_*(\F)$.
\end{Definition}

\begin{Lemma}[X. Sun \cite{Sun08i}]\label{Lem:Sun}
Let k be an algebraically closed field of characteristic $p>0$, $X$ a smooth projective curve of genus $g\geq 2$ over $k$, and $\E$ a vector bundle on $X$. Then the canonical filtration of $F^*_X{F_X}_*(\E)$ is $$0=V_p\subset V_{p-1}\subset\cdots\subset V_{l+1}\subset V_l\subset\cdots\subset V_1\subset V_0=F^*_X{F_X}_*(\E)$$
such that
\begin{itemize}
\item[$(1)$] $\nabla_{\mathrm{can}}(V_{i+1})\subset V_i\otimes_{\Ox_X}\Omg^1_X$ for $0\leq i\leq p-1$.
\item[$(2)$] $V_l/V_{l+1}\stackrel{\nabla_{\mathrm{can}}}{\rightarrow}\E\otimes_{\Ox_X}\Omg^{\otimes l}_X$ is isomorphic for $0\leq l\leq p-1$.
\item[$(3)$] If $g\geq 2$ and $\E$ is semistable, then the canonical filtration of $F^*_X{F_X}_*(\E)$ is nothing but the Harder-Narasimhan filtration of $F^*_X{F_X}_*(\E)$.
\item[$(4)$] If $g\geq 1$, then ${F_X}_*(\E)$ is semistable whenever $\E$ is semistable. If $g\geq 2$, then ${F_X}_*(\E)$ is stable whenever $\E$ is stable.
\end{itemize}
\end{Lemma}

\begin{Proposition}\label{Prop:Injection}
Let $k$ be an algebraically closed field of characteristic $3$, and $X$ a smooth projective curve of genus $2$ over $k$. Let $\E$ be a stable vector bundle on $X$ with $\rk(\E)=3$, $\deg(\E)=d$ and one has non-trivial homomorphism $F^*_X(\E)\rightarrow\Ls$, where $\Ls$ is a line bundle with $\deg(\Ls)=d-1$. Then the adjoint homomorphism $\E\hookrightarrow{F_X}_*(\Ls)$ is an injection.
\end{Proposition}

\begin{proof}
By adjunction, there is a nontrivial homomorphism $\E\rightarrow{F_X}_*(\Ls)$. Denote the image by $\G$. Suppose that $1\leq\rk(\G)\leq 2$, then by \cite[Corollary 2.4]{Sun10ii} and the stability of ${F_X}_*(\Ls)$, we have
$$\mu(\G)-\mu({F_X}_*(\Ls))<-\frac{p-\rk(\G)}{p}(g-1)=-\frac{3-\rk(\G)}{3}.$$
By Grothendieck-Riemann-Roch theorem, we have $\deg({F_X}_*(\Ls))=d+1$ (cf. \cite[Lemma 4.2]{Sun08i}), so $$\mu(\G)<-\frac{3-\rk(\G)}{3}+\mu({F_X}_*(\Ls))=\frac{\rk(\G)+d-2}{3}.$$
On the other hand, by stability of $\E$, we have $\mu(\G)>\frac{d}{3}$. As $1\leq\rk(\G)\leq 2$, this is impossible. Hence, $\rk(\G)=3$.
Therefore $\E\cong\G$, i.e. the adjoint homomorphism $\E\hookrightarrow{F_X}_*(\Ls)$ is an injection.
\end{proof}

\begin{Proposition}\label{Prop:Subsheaf}
Let $k$ be an algebraically closed field of characteristic $3$, and $X$ a smooth projective curve of genus $2$ over $k$. Let $\Ls$ be a line bundle on $X$ with $\deg(\Ls)=d-1$, $\E$ a subsheaf of $F^*_X(\Ls)$ with $\rk(\E)=3$ and $\deg(\E)=d$. Then $\E$ is a stable vector bundle.
\end{Proposition}

\begin{proof}
Let $\G\subset\E$ be a subsheaf of $\E$ with $\rk(\G)<\rk(\E)=3$. By \cite[Corollary 2.4]{Sun10ii} and the stability of ${F_X}_*(\Ls)$, we have
$$\mu(\G)-\mu({F_X}_*(\Ls))<-\frac{p-\rk(\G)}{p}(g-1)=-\frac{3-\rk(\G)}{3}.$$
It follows that
$$\mu(\G)<-\frac{3-\rk(\G)}{3}+\mu({F_X}_*(\Ls))=\frac{\rk(\G)+d-2}{3}\leq\frac{d}{3}=\mu(\E).$$
Thus $\E$ is a stable vector bundle.
\end{proof}

By Proposition \ref{Prop:Injection}, Proposition \ref{Prop:Subsheaf} and the classification of Harder-Narasimhan polygons of Frobenius pull backs of Frobenius destabilized stable vector bundles in the case $(p,g,r,d)=(3,2,3,d)$, we know that any Frobenius destabilized stable bundle $[\E]\in\M^s_{X}(3,d)$ with $\HNP(F^*_X(\E))\in\{\Ps_2(d),\Ps_3(d),\Ps_4(d)\}$ can be embedded into ${F_X}_*(\Ls)$ for some line bundle $\Ls$ of degree $d-1$.

\section{Geometric Properties of Quot Schemes}

We first recall some construction of Quot schemes in \cite[Section 4]{Li18}. Let $k$ be an algebraically closed field of characteristic $p>0$, $X$ a smooth projective curve of genus $g$ over $k$, $F_X:X\rightarrow X$ be the absolute Frobenius morphism. Let $\Pic^{(t)}(X)$ be the Picard scheme parameterizes all line bundles of degree $t$ on $X$. Using the same notations in \cite[Section 4]{Li18}, we relist the Quot schemes as following
\begin{eqnarray*}
\Quot_X(r,d,\Pic^{(t)}(X))&:=&\{~[\E\hookrightarrow{F_X}_*(\Ls)]~|~\rk(\E)=r,\deg(\E)=d,\Ls\in\Pic^{(t)}(X)~\}\\
\Quot_X(r,d,\Pic^{(t)}(X),\Ps)&:=&\{~[\E\hookrightarrow{F_X}_*(\Ls)]\in\Quot_X(r,d,\Pic^{(t)}(X))~|~\HNP(F^*_X(\E))=\Ps~\}.\\
\Quot_X(r,d,\Pic^{(t)}(X),\Ps^+)&:=&\{~[\E\hookrightarrow{F_X}_*(\Ls)]\in\Quot_X(r,d,\Pic^{(t)}(X))~|~\HNP(F^*_X(\E))\succcurlyeq\Ps~\}\\
\Quot^\sharp_X(r,d,\Pic^{(t)}(X))&:=&\left\{~[\E\hookrightarrow{F_X}_*(\Ls)]~\Bigg|
\begin{array}{lll}
\rk(\E)=r,\deg(\E)=d,\Ls\in\Pic^{(t)}(X),\\
\text{adjoint homomorphism}~F^*_X(\E)\rightarrow\Ls\\
\text{is surjective}.
\end{array}
\right\}\\
\Quot^\sharp_X(r,d,\Pic^{(t)}(X),\Ps)&:=&\{~[\E\hookrightarrow{F_X}_*(\Ls)]\in\Quot^\sharp_X(r,d,\Pic^{(t)}(X))~|~\HNP(F^*_X(\E))=\Ps~\}.\\
\Quot^\sharp_X(r,d,\Pic^{(t)}(X),\Ps^+)&:=&\{~[\E\hookrightarrow{F_X}_*(\Ls)]\in\Quot^\sharp_X(r,d,\Pic^{(t)}(X))~|~\HNP(F^*_X(\E))\succcurlyeq\Ps~\}\\
\Quot_X(r,d,\Ls)&:=&\{~[\E\hookrightarrow{F_X}_*(\Ls)]~|~\rk(\E)=r,\deg(\E)=d~\}\\
\Quot_X(r,d,\Ls,\Ps)&:=&\{~[\E\hookrightarrow{F_X}_*(\Ls)]\in\Quot_X(r,d,\Ls)~|~\HNP(F^*_X(\E))=\Ps~\}\\
\Quot_X(r,d,\Ls,\Ps^+)&:=&\{~[\E\hookrightarrow{F_X}_*(\Ls)]\in\Quot_X(r,d,\Ls)~|~\HNP(F^*_X(\E))\succcurlyeq\Ps~\}\\
\Quot^\sharp_X(r,d,\Ls)&:=&\left\{~[\E\hookrightarrow{F_X}_*(\Ls)]~\Bigg|
\begin{array}{lll}
\rk(\E)=r,\deg(\E)=d,~\text{adjoint}\\
\text{homomorphism}~F^*_X(\E)\rightarrow\Ls~\text{is surjective}.
\end{array}
\right\}\\
\Quot^\sharp_X(r,d,\Ls,\Ps)&:=&\{~[\E\hookrightarrow{F_X}_*(\Ls)]\in\Quot^\sharp_X(r,d,\Ls)~|~\HNP(F^*_X(\E))=\Ps~\}\\
\Quot^\sharp_X(r,d,\Ls,\Ps^+)&:=&\{~[\E\hookrightarrow{F_X}_*(\Ls)]\in\Quot^\sharp_X(r,d,\Ls)~|~\HNP(F^*_X(\E))\succcurlyeq\Ps~\}\\
\end{eqnarray*}

In this section, we are interesting in the Frobenius stratification of moduli spaces of vector bundles in the case $(p,g,r,d,t)=(3,2,3,d,d-1)$.
In this case, the scheme $\Quot_X(3,d,\Pic^{(d-1)}(X))$ parameterizes the subsheaves of ${F_X}_*(\Ls)$ with rank $3$ and degree $d$ for any line bundle $\Ls$ of degree $d-1$ on $X$. By Proposition \ref{Prop:Injection} and Proposition \ref{Prop:Subsheaf}, we know that these vector bundles are stable. This induces a natural morphism
$$\theta:\Quot_X(3,d,\Pic^{(d-1)}(X))\rightarrow\M^s_X(3,d):[\E\hookrightarrow{F_X}_*(\Ls)]\mapsto[\E]$$
Now, we analysis the structure of the Quot scheme $\Quot_X(3,d,\Pic^{(d-1)}(X))$. Let $[\E\hookrightarrow{F_X}_*(\Ls)]$ be a closed point of $\Quot_X(3,d,\Pic^{(d-1)}(X))$, where $\Ls\in\Pic^{(d-1)}(X)$. The non-trivial adjoint homomorphism $F^*_X(\E)\rightarrow\Ls$ implies that $$\mu(F^*_X(\E))>\mu(\Ls)\geq\mu_{\mathrm{min}}(F^*_X(\E)),$$ so $\E$ is a Frobenius destabilized stable vector bundle.

\begin{Lemma}[S. S. Shatz \cite{Shatz77} Theorem 2 and Theorem 3]\label{Shatz77}
Let $k$ be an algebraically closed field, $X$ a smooth projective variety over $k$, $H$ an ample divisor on $X$. Consider the Harder-Narasimhan filtrations of torsion free sheaves on $X$ in the sense of Mumford's semistability. Then
\begin{itemize}
\item[$(1)$] For any torsion free sheaf $\E$ and any subsheaf $\F\subseteq\E$, we have the point $(\rk(\F),\deg(\F))$ lies below $\HNP(\E)$.
\item[$(2)$] Let $\Ec$ be a flat family of torsion free sheaves of rank $r$ and degree $d$ on $S\times_kX$ parameterized by a scheme $S$ of finite type over $k$. Then for any convex polygon $\Ps\in\ConPgn(r,d)$, the subset $$S_{\Ps}=\{s\in S~|~\HNP(\Ec|_s)\succcurlyeq\Ps\}$$ is a closed subset of $S$.
\end{itemize}
\end{Lemma}

\begin{Proposition}\label{Prop:Intersection}
Let $k$ be an algebraically closed field of characteristic $3$, and $X$ a smooth projective curve of genus $2$ over $k$. Let $\Ls$ be a line bundle on $X$ with $\deg(\Ls)=d-1$, $0\subset E_2\subset E_1\subset F^*_X({F_X}_*(\Ls))$ the canonical filtration of $F^*_X({F_X}_*(\Ls))$. Let $[\E\hookrightarrow{F_X}_*(\Ls)]\in\Quot_X(3,d,\Ls)$. Then $\HNP(F^*_X(\E))\in\{\Ps_2(d),\Ps_3(d),\Ps_4(d)\}$, and
\begin{itemize}
\item[$(1)$] $\HNP(F^*_X(\E))=\Ps_4(d)$ if and only if $\deg(F^*_X(\E)\cap E_2)=d+2$ if and only if the adjoint homomorphism $F^*_X(\E)\rightarrow\Ls$ is not surjective. In this case, the Harder-Narasimhan filtration of $F^*_X(\E)$ is $$0\subset F^*_X(\E)\cap E_2\subset F^*_X(\E)\cap E_1\subset F^*_X(\E).$$
\item[$(2)$] $\HNP(F^*_X(\E))=\Ps_3(d)$ if and only if $\deg(F^*_X(\E)\cap E_2)=d+1$. In this case, $[\E\hookrightarrow{F_X}_*(\Ls)]\in\Quot^\sharp_X(3,d,\Ls)$ and the Harder-Narasimhan filtration of $F^*_X(\E)$ is $$0\subset F^*_X(\E)\cap E_2\subset F^*_X(\E)\cap E_1\subset F^*_X(\E).$$
\item[$(3)$] $\HNP(F^*_X(\E))=\Ps_2(d)$ if and only if $\deg(F^*_X(\E)\cap E_2)=d$. In this case, $[\E\hookrightarrow{F_X}_*(\Ls)]\in\Quot^\sharp_X(3,d,\Ls)$ and the Harder-Narasimhan filtration of $F^*_X(\E)$ is $$0\subset F^*_X(\E)\cap E_1\subset F^*_X(\E).$$
\end{itemize}
\end{Proposition}

\begin{proof}
The canonical filtration $0\subset E_2\subset E_1\subset F^*_X({F_X}_*(\Ls))$ induces the filtration $$0\subset F^*_X(\E)\cap E_2\subset F^*_X(\E)\cap E_1\subset F^*_X(\E).$$
Then the injections $F^*_X(\E)\cap E_2\hookrightarrow E_2$ and $$F^*_X(\E)/(F^*_X(\E)\cap E_1)\hookrightarrow F^*_X({F_X}_*(\Ls))/E_1\cong\Ls$$ imply that $\deg(F^*_X(\E)\cap E_2)\leq d+3$ and $\deg(F^*_X(\E)/(F^*_X(\E)\cap E_1))\leq d-1$. Suppose that $[\E\hookrightarrow{F_X}_*(\Ls)]\in\Quot_X(3,d,\Ls)$ such that $F^*_X(\E)$ is semistable or $\HNP(F^*_X(\E))=\Ps_1(d)$, then $\mu_{\mathrm{min}}(F^*_X(\E))\geq\frac{2d-1}{2}$. This contradicts the fact $$\deg(F^*_X(\E)/(F^*_X(\E)\cap E_1))\leq d-1.$$ Hence $\HNP(F^*_X(\E))\in\{\Ps_2(d),\Ps_3(d),\Ps_4(d)\}$.

We first claim that $$d\leq\deg(F^*_X(\E)\cap E_2)\leq d+2.$$ The fact $\deg(F^*_X(\E)\cap E_2)\leq d+2$ followed by Lemma \ref{Shatz77}(1) and the classification of Harder-Narasimhan polygons of Frobenius pull backs of Frobenius destabilized stable vector bundles in the case $(p,g,r,d)=(3,2,3,d)$. On the other hand, suppose that $\deg(F^*_X(\E)\cap E_2)\leq d-1$. Since $\deg(F^*_X(\E)/(F^*_X(\E)\cap E_1))\leq d-1$, we have $$\deg((F^*_X(\E)\cap E_1)/(F^*_X(\E)\cap E_2))\geq d+2.$$ Then the injection
$$(F^*_X(\E)\cap E_1)/(F^*_X(\E)\cap E_2)\hookrightarrow(F^*_X(\E)/(F^*_X(\E)\cap E_1))\otimes_{\Ox_X}\Omg^1_X$$ induces a contradiction, since $\deg((F^*_X(\E)/(F^*_X(\E)\cap E_1))\otimes_{\Ox_X}\Omg^1_X)\leq d+1$. This completes the proof of the claim.

(1).
I. If $\HNP(F^*_X(\E))=\Ps_4(d)$, there exists unique maximal destabilizing sub-line bundle $E'\subset F^*_X(\E)$ with $\deg(E')=d+2$. Suppose that $E'\nsubseteq F^*_X(\E)\cap E_1$, then the composition $$E'\hookrightarrow F^*_X(\E)\hookrightarrow F^*_X({F_X}_*(\Ls))\twoheadrightarrow F^*_X({F_X}_*(\Ls))/E_1\cong\Ls$$ is not trivial. This induces a contradiction since $\deg(E')>\deg(\Ls)$. Suppose that $E'\subset F^*_X(\E)\cap E_1$ and $E'\nsubseteq F^*_X(\E)\cap E_2$, then the composition $$E'\hookrightarrow F^*_X(\E)\cap E_1\hookrightarrow E_1\twoheadrightarrow E_1/E_2$$ is not trivial. This induces a contradiction since $\deg(E')>\deg(E_1/E_2)$. Hence $E'\subset F^*_X(\E)\cap E_2$. In fact $E'=F^*_X(\E)\cap E_2$. Thus $\deg(F^*_X(\E)\cap E_2)=d+2$.

II. If $\deg(F^*_X(\E)\cap E_2)=d+2$, then the injection $$F^*_X(\E)\cap E_2\hookrightarrow((F^*_X(\E)\cap E_1)/(F^*_X(\E)\cap E_2))\otimes_{\Ox_X}\Omg^1_X$$ implies that $\deg((F^*_X(\E)\cap E_1)/(F^*_X(\E)\cap E_2))\geq d$. So $$\deg(F^*_X(\E)/(F^*_X(\E)\cap E_1))\leq d-2<\deg(\Ls).$$ Hence the adjoint homomorphism $F^*_X(\E)\rightarrow\Ls$ is not surjective.

III. If $F^*_X(\E)\rightarrow\Ls$ is not surjective, then $\mu_{\mathrm{min}}(F^*_X(\E))\leq d-2$. Then we must have $\HNP(F^*_X(\E))=\Ps_4(d)$ by Lemma \ref{Shatz77}(1) and the classification of Harder-Narasimhan polygons of Frobenius pull backs of Frobenius destabilized stable vector bundles in the case $(p,g,r,d)=(3,2,3,d)$.

In this case, it is easy to see that the Harder-Narasimhan filtration $F^*_X(\E)$ is $$0\subset F^*_X(\E)\cap E_2\subset F^*_X(\E)\cap E_1\subset F^*_X(\E).$$

(2). I. If $\deg(F^*_X(\E)\cap E_2)=d+1$, then the adjoint homomorphism $F^*_X(\E)\rightarrow\Ls$ is surjective by (1) and $$\mu(F^*_X(\E)\cap E_2)>\mu((F^*_X(\E)\cap E_1)/(F^*_X(\E)\cap E_2))>\mu(F^*_X(\E)/(F^*_X(\E)\cap E_1)).$$ Hence,
$\HNP(F^*_X(\E))=\Ps_3(d)$ and the Harder-Narasimhan filtration of $F^*_X(\E)$ is $$0\subset F^*_X(\E)\cap E_2\subset F^*_X(\E)\cap E_1\subset F^*_X(\E).$$

II. If $\HNP(F^*_X(\E))=\Ps_3(d)$, there exists unique maximal destabilizing sub-line bundle $E'\subset F^*_X(\E)\cap E_1$ with $\deg(E')=d+1$. Suppose that $E'\nsubseteq F^*_X(\E)\cap E_2$, then $$E'\hookrightarrow F^*_X(\E)\cap E_1\hookrightarrow E_1\twoheadrightarrow E_1/E_2$$ is not trivial. This implies $E'\cong E_1/E_2$ since $E'$ and $E_1/E_2$ are line bundles with same degree. Then $E_1=E'\oplus E''$ for some line bundle $E''$ of degree $d+3$. This induces a contradiction by Corollary \ref{Cor:NonSplitting}. Hence $E'\subseteq F^*_X(\E)\cap E_2$. In fact $E'=F^*_X(\E)\cap E_2$, so $\deg(F^*_X(\E)\cap E_2)=d+1$.

(3). By the proof of (1) and (2), we can conclude that $\HNP(F^*_X(\E))=\Ps_2(d)$ if and only if $\deg(F^*_X(\E)\cap E_2)=d$. In this case, $[\E\hookrightarrow{F_X}_*(\Ls)]\in\Quot^\sharp_X(3,d,\Ls)$ and the Harder-Narasimhan filtration $F^*_X(\E)$ is $$0\subset F^*_X(\E)\cap E_1\subset F^*_X(\E).$$
\end{proof}

\begin{Lemma}[A. Grothendieck, M. Raynaud]\label{Lem:Grothendieck}
Let $k$ be an algebraically closed field of characteristic $p>2$, $X$ a smooth projective curve of genus $g\geq 2$ over $k$. Let $B^1_X$ be the locally free sheaf of locally exact differential forms on $X$ defined by the exact sequence of locally free sheaves
$$0\rightarrow\mathscr{O}_X\rightarrow{F_X}_*(\mathscr{O}_X)\rightarrow B^1_X\rightarrow 0.$$
Then the Harder-Narasimhan filtration of $F^*_X(B^1_X)$ is $$0=V_p\subset V_{p-1}\subset\cdots\subset V_{l+1}\subset V_l\subset\cdots\subset V_1=F^*_X(B^1_X)$$
such that $V_i/V_{i+1}\cong\mathrm{\Omega}^{\otimes i}_X$ for any $1\leq i\leq p-1$, and $p|(g-1)$ if and only if $$F^*_X(B^1_X)\ncong\mathrm{\Omega}^{\otimes p-1}_X\oplus\mathrm{\Omega}^{\otimes p-2}_X\oplus\cdots\oplus\mathrm{\Omega}^1_X.$$
In particular, in the case $p=3$ and $g=2$, we have $$F^*_X(B^1_X)\ncong\mathrm{\Omega}^{\otimes 2}_X\oplus\mathrm{\Omega}^1_X.$$
\end{Lemma}

By Lemma \ref{Lem:Grothendieck}, it is easily to deduce the following corollary.

\begin{Corollary}\cite[Corollary 4.4]{Li18}\label{Cor:NonSplitting}
Let $k$ be an algebraically closed field of characteristic $p>2$, $X$ a smooth projective curve of genus $g\geq 2$ over $k$, and $\Ls$ a line bundle on $X$. Then the Harder-Narasimhan filtration of $F^*_X{F_X}_*(\Ls)$ is $$0=V_p\subset V_{p-1}\subset\cdots\subset V_{l+1}\subset V_l\subset\cdots\subset V_1\subset V_0=F^*_X{F_X}_*(\Ls)$$
such that $V_i/V_{i+1}\cong\mathrm{\Omega}^{\otimes i}_X\otimes\Ls$ for any $0\leq i\leq p-1$, and $p|(g-1)$ if and only if $$F^*_X{F_X}_*(\Ls)\ncong(\mathrm{\Omega}^{\otimes p-1}_X\otimes\Ls)\oplus(\mathrm{\Omega}^{\otimes p-2}_X\otimes\Ls)\oplus\cdots\oplus(\mathrm{\Omega}^1_X\otimes\Ls)\oplus\Ls.$$
In particular, in the case $p=3$ and $g=2$, we have $$F^*_X{F_X}_*(\Ls)\ncong(\mathrm{\Omega}^{\otimes 2}_X\otimes\Ls)\oplus(\mathrm{\Omega}^1_X\otimes\Ls)\oplus\Ls.$$
\end{Corollary}

\begin{Proposition}\label{DimofQuot}
Let $k$ be an algebraically closed field of characteristic $3$, and $X$ a smooth projective curve of genus $2$ over $k$. Then $\Quot_X(3,d,\Pic^{(d-1)}(X),\Ps^+_i(d))$ are smooth irreducible projective varieties for $2\leq i\leq 4$, and
$$\dim\Quot_X(3,d,\Pic^{(d-1)}(X),\Ps_i(d))=\dim\Quot_X(3,d,\Pic^{(d-1)}(X),\Ps^+_i(d))=
\begin{cases}
5, i=2\\
4, i=3\\
3, i=4\\
\end{cases}$$
\end{Proposition}

\begin{proof}
By \cite{Grothendieck95}, there is a morphism
\begin{eqnarray*}
\Pi:\Quot_X(3,d,\Pic^{(d-1)}(X))&\rightarrow&X\times\Pic^{(d-1)}(X)\\
{[\E\hookrightarrow{F_X}_*(\Ls)]}&\mapsto&(\Supp({F_X}_*(\Ls)/\E),\Ls).
\end{eqnarray*}
For any point $x\in X$ and any $[\Ls]\in\Pic^{(d-1)}(X)$, the fiber of $\Pi$ over $(x,[\Ls])$ denoted by $\Quot_X(3,d,\Ls_x)$. Then there is an one to one correspondence between the set of closed points $[\E\hookrightarrow{F_X}_*(\Ls)]$ of $\Quot_X(3,d,\Ls_x)$ and the set of $\Ox_x$-submodules $V\subset{F_X}_*(\Ls)_x$ such that ${F_X}_*(\Ls)_x/V\cong k$, the latter has a natural structure of algebraic variety which is isomorphic to projective space $\PP^2_k$. Hence $\Quot_X(3,d,\Pic^{(d-1)}(X))$ is a smooth irreducible projective variety of dimension $5$. Without loss of generality, we can assume that $\Ox_x\cong k[[t^3]]$, then ${F_X}_*(\Ls)_x\cong k[[t]]$ endows with $k[[t^3]]$-module structure induced by injection $k[[t^3]]\hookrightarrow k[[t]]$ and $$F^*_X({F_X}_*(\Ls))_x\cong k[[t]]\otimes_{k[[t^3]]}k[[t]].$$

Suppose that the $\Ox_x$-submodule $\E_x$ of ${F_X}_*(\Ls)_x$ corresponds to the $k[[t^3]]$-submodule $V_{\E}$ of $k[[t]]$, then the $\Ox_x$-submodule $F^*_X(\E)_x$ of $F^*_X({F_X}_*(\Ls))_x$ corresponds to the $k[[t]]$-submodule $V_{\E}\otimes_{k[[t^3]]}k[[t]]$ of $k[[t]]\otimes_{k[[t^3]]}k[[t]]$.

Consider the decomposition of $k[[t]]=k[[t^3]]\oplus k[[t^3]]\cdot t\oplus k[[t^3]]\cdot t^2$ as $k[[t^3]]$-module. Then the $k[[t^3]]$-submodule $V_{\E}\subset k[[t]]$ with $k[[t]]/V_{\E}\cong k$ implies that $$k[[t^3]]\cdot t^3\oplus k[[t^3]]\cdot t^4\oplus k[[t^3]]\cdot t^5 \subset V_{\E}.$$

Now, we investigate the intersection of $F^*_X(\E)$ with the canonical filtration $$0\subset E_2\subset E_1\subset F^*_X({F_X}_*(\Ls)).$$
Locally, the stalk ${E_1}_x$ has a basis $\{t\otimes 1-1\otimes t, (t\otimes 1-1\otimes t)^2\}$ and ${E_2}_x$ has a basis $\{(t\otimes 1-1\otimes t)^2\}$ as $k[[t]]$-submodules of $F^*_X({F_X}_*(\Ls))_x\cong k[[t]]\otimes_{k[[t^3]]}k[[t]]$ by \cite[Lemma 3.2]{Sun08i}. Let $[\E\hookrightarrow{F_X}_*(\Ls)]\in\Quot_X(3,d,\Ls_x)$, we claim that
\begin{itemize}
\item[$(a)$] $(t\otimes 1-1\otimes t)^2\notin V_{\E}\otimes_{k[[t^3]]}k[[t]]$
\item[$(b)$] $(t\otimes 1-1\otimes t)^2t\in V_{\E}\otimes_{k[[t^3]]}k[[t]]$ if and only if $\{t,t^2\}\subset V_{\E}$.
\item[$(c)$] $(t\otimes 1-1\otimes t)^2t^2\in V_{\E}\otimes_{k[[t^3]]}k[[t]]$ if and only if $t^2\in V_{\E}$.
\item[$(d)$] $(t\otimes 1-1\otimes t)^2t^3\in V_{\E}\otimes_{k[[t^3]]}k[[t]]$.
\end{itemize}

Suppose that $(t\otimes 1-1\otimes t)^2\in V_{\E}\otimes_{k[[t^3]]}k[[t]]$, then we have $$V_{\E}\otimes_{k[[t^3]]}k[[t]]=(t\otimes 1-1\otimes t)^2k[[t]].$$ It follows that $F^*_X(\E)\cap E_2=E_2$ , so $\deg(F^*_X(\E)\cap E_2)=d+3$. This contradicts to Proposition \ref{Prop:Intersection}.

Since
\begin{eqnarray*}
(t\otimes 1-1\otimes t)^2t&=&t^2\otimes t-2t\otimes t^2+1\otimes t^3\\
&=&t^2\otimes t-2t\otimes t^2+t^3\otimes 1
\end{eqnarray*}
and $\{t^3\}\subset V_{\E}$ by $k[[t^3]]\cdot t^3\oplus k[[t^3]]\cdot t^4\oplus k[[t^3]]\cdot t^5 \subset V_{\E}$. Hence, $$(t\otimes 1-1\otimes t)^2t\in V_{\E}\otimes_{k[[t^3]]}k[[t]]~\mathrm{iff}~t^2\otimes t-2t\otimes t^2\in V_{\E}\otimes_{k[[t^3]]}k[[t]],$$ that is equivalent to $\{t,t^2\}\subset V_{\E}$.

Since
\begin{eqnarray*}
(t\otimes 1-1\otimes t)^2t^2&=&t^2\otimes t^2-2t\otimes t^3+1\otimes t^4\\
&=&t^2\otimes t^2-2t^4\otimes 1+t^3\otimes t
\end{eqnarray*}
and $\{t^3,t^4\}\subset V_{\E}$ by $k[[t^3]]\cdot t^3\oplus k[[t^3]]\cdot t^4\oplus k[[t^3]]\cdot t^5 \subset V_{\E}$. Hence, $$(t\otimes 1-1\otimes t)^2t^2\in V_{\E}\otimes_{k[[t^3]]}k[[t]]~\mathrm{iff}~t^2\otimes t^2\in V_{\E}\otimes_{k[[t^3]]}k[[t]],$$ that is equivalent to $x^2\in V_{\E}$..

Since
\begin{eqnarray*}
(t\otimes 1-1\otimes t)^2t^3&=&t^2\otimes t^3-2t\otimes t^4+1\otimes t^5\\
&=&t^5\otimes 1-2t^4\otimes t+t^3\otimes t^2
\end{eqnarray*}
and $\{t^3,t^4,t^5\}\subset V_{\E}$ by $k[[t^3]]\cdot t^3\oplus k[[t^3]]\cdot t^4\oplus k[[t^3]]\cdot t^5 \subset V_{\E}$. It follows that $$(t\otimes 1-1\otimes t)^2t^3\in V_{\E}\otimes_{k[[t^3]]}k[[t]].$$

In summary, by above claim, we have $$1\leq\dim {E_2}_x/((V_{\E}\otimes_{k[[t^3]]}k[[t]])\cap {E_2}_x)\leq 3.$$ In detail,
$$\dim {E_2}_x/((V_{\E}\otimes_{k[[t^3]]}k[[t]])\cap {E_2}_x)=
\begin{cases}
1& \text{if ond only if}~\{t,t^2\}\subset V_{\E}\\
2& \text{if ond only if}~t\notin V_{\E}~\text{and}~t^2\in V_{\E}\\
3& \text{if ond only if}~t^2\notin V_{\E}
\end{cases}$$

Consider the exact sequence of $\Ox_X$-modules
$$0\rightarrow F^*_X(\E)\cap E_2\rightarrow E_2\rightarrow E_2/(F^*_X(\E)\cap E_2)\rightarrow0.$$
Notice that $E_2/(F^*_X(\E)\cap E_2)={E_2}_x/((V_{\E}\otimes_{k[[t^3]]}k[[t]])\cap{E_2}_x)$. Therefore, by Proposition \ref{Prop:Intersection}, we have
\begin{eqnarray*}
\HNP(F^*_X(\E))=\Ps_2(d) &\Leftrightarrow& \deg(F^*_X(\E)\cap E_2)=d\\
&\Leftrightarrow& \deg({E_2}_x/((V_{\E}\otimes_{k[[t^3]]}k[[t]])\cap{E_2}_x))=3\\
&\Leftrightarrow& t^2\notin V_{\E}.
\end{eqnarray*}

\begin{eqnarray*}
\HNP(F^*_X(\E))=\Ps_3(d) &\Leftrightarrow& \deg(F^*_X(\E)\cap E_2)=d+1\\
&\Leftrightarrow& \deg({E_2}_x/((V_{\E}\otimes_{k[[t^3]]}k[[t]])\cap{E_2}_x))=2\\
&\Leftrightarrow& t\notin V_{\E}~\text{and}~t^2\in V_{\E}.
\end{eqnarray*}

\begin{eqnarray*}
\HNP(F^*_X(\E))=\Ps_4(d) &\Leftrightarrow& \deg(F^*_X(\E)\cap E_2)=d+2\\
&\Leftrightarrow& \deg({E_2}_x/((V_{\E}\otimes_{k[[t^3]]}k[[t]])\cap{E_2}_x))=1\\
&\Leftrightarrow& \{t,t^2\}\subset V_{\E}.
\end{eqnarray*}

For $i=2,3,4$, let $$\Quot_X(3,d,\Ls_x,\Ps^+_i(d)):=\{~[\E\hookrightarrow{F_X}_*(\Ls)]\in\Quot_X(3,d,\Ls_x)~|~\HNP(F^*_X(\E))\succcurlyeq\Ps_i(d)~\}$$
Then
\begin{eqnarray*}
\Quot_X(3,d,\Ls_x,\Ps^+_2(d))&\cong&\{V|~k[[t^3]]\text{-submodule}~V\subset k[[t]],k[[t]]/V\cong k\}\cong\PP^2_k\\
\Quot_X(3,d,\Ls_x,\Ps^+_3(d))&\cong&\{V|~k[[t^3]]\text{-submodule}~V\subset k[[t]],k[[t]]/V\cong k,t^2\in V_{\E}\}\cong\PP^1_k\\
\Quot_X(3,d,\Ls_x,\Ps^+_4(d))&\cong&\{V|~k[[t^3]]\text{-submodule}~V\subset k[[t]],k[[t]]/V\cong k,\{t,t^2\}\subset V_{\E}\}\cong\{p\}.
\end{eqnarray*}
So $\Quot_X(3,d,\Pic^{(d-1)}(X),\Ps^+_i(d))$ are smooth irreducible projective varieties for $2\leq i\leq 4$, and
$$\dim\Quot_X(3,d,\Pic^{(d-1)}(X),\Ps^+_i(d))=
\begin{cases}
5& i=2\\
4& i=3\\
3& i=4\\
\end{cases}$$

By Lemma \ref{Shatz77}(2), it is easy to see that $\Quot_X(3,d,\Pic^{(d-1)}(X),\Ps_i(d))$ is an open subvariety of $\Quot_X(3,d,\Pic^{(d-1)}(X),\Ps^+_i(d))$ for $2\leq i\leq 4$. Therefore, $$\overline{\Quot_X(3,d,\Pic^{(d-1)}(X),\Ps_i(d))}=\Quot_X(3,d,\Pic^{(d-1)}(X),\Ps^+_i(d)),$$ $$\dim\Quot_X(3,d,\Pic^{(d-1)}(X),\Ps_i(d))=\dim\Quot_X(3,d,\Pic^{(d-1)}(X),\Ps^+_i(d))$$ for $2\leq i\leq 4$.
\end{proof}

\section{Geometric Properties of Frobenius strata}\

We now study the geometric properties of strata in the Frobenius stratification of moduli space $\M^s_X(3,d)$, when $X$ is a smooth projective curve of genus $2$ over an algebraically closed field $k$ of characteristic $3$.

\begin{Proposition}\label{Prop:InjMorphism}
Let $k$ be an algebraically closed field of characteristic $3$, and $X$ a smooth projective curve of genus $2$ over $k$. Then the image of the morphism
\begin{eqnarray*}
\theta:\Quot_X(3,d,\Pic^{(d-1)}(X))&\rightarrow&\M^s_X(3,d)\\
{[\E\hookrightarrow{F_X}_*(\Ls)]}&\mapsto&[\E]
\end{eqnarray*}
is the subset
$$\{~[\E]\in\M^s_X(3,d)~|~\HNP(F^*_X(\E))\in\{\Ps_2(d),\Ps_3(d),\Ps_4(d)\}~\}.$$
Moreover, the restriction $\theta|_{\Quot^\sharp_X(3,d,\Pic^{(d-1)}(X))}$ is an injective morphism and the image of $\theta|_{\Quot^\sharp_X(3,d,\Pic^{(d-1)}(X))}$ is the subset
$$\{~[\E]\in\M^s_X(3,d)~|~\HNP(F^*_X(\E))\in\{\Ps_2(d),\Ps_3(d)\}~\}.$$
\end{Proposition}

\begin{proof}
Let $[\E]\in\Quot_X(3,d,\Ls)$, then $\HNP(F^*_X(\E))\in\{\Ps_2(d),\Ps_3(d),\Ps_4(d)\}$ by Proposition \ref{Prop:Intersection}. It follows that the image of $\theta$ locus in the following subset $$\{~[\E]\in\M^s_X(3,d)~|~\HNP(F^*_X(\E))\in\{\Ps_2(d),\Ps_3(d),\Ps_4(d)\}~\}.$$

On the other hand, let $[\E]\in\M^s_X(3,d)$ such that $\HNP(F^*_X(\E))=\Ps_i(d)$ for some $2\leq i\leq 4$. Then $F^*_X(\E)$ has a quotient line bundle $\Ls'$ of $\deg(\Ls')\leq d-1$. Embedding $\Ls'$ into some line bundle $\Ls$ of $\deg(\Ls)=d-1$. So we get the non-trivial homomorphism $$F^*_X(\E)\twoheadrightarrow\Ls'\hookrightarrow\Ls.$$ Then the adjunction $\E\hookrightarrow{F_X}_*(\Ls)$ is an injection by Proposition \ref{Prop:Injection}. Hence, the image of $\theta$ is $$\{~[\E]\in\M^s_X(3,d)~|~\HNP(F^*_X(\E))\in\{\Ps_2(d),\Ps_3(d),\Ps_4(d)\}~\}.$$

Now, we will prove $\theta|_{\Quot^\sharp_X(3,d,\Pic^{(d-1)}(X))}$ is an injective morphism. Let $$e_i:=[\E_i\hookrightarrow{F_X}_*(\Ls_i)]\in\Quot^\sharp_X(3,d,\Pic^{(d-1)}(X)),$$ where $\Ls_i\in\Pic^{(d-1)}(X)$, $i=1,2$. Suppose that $\theta(e_1)=\theta(e_2)\in\M^s_X(3,d)$, i.e. $\E_1\cong\E_2$. Since $\HNP(F^*_X(\E))\in\{\Ps_2(d),\Ps_3(d)\}$, we have $\mu_{\mathrm{min}}(F^*_X(\E_i))=d-1$. So the surjection $F^*_X(\E_i)\rightarrow\Ls_i$ implies that $\Ls_i$ is the quotient line bundle of $F^*_X(\E_i)$ with minimal slope in the Harder-Narasimhan filtration of $F^*_X(\E_i)$. By the uniqueness of Harder Narasimhan filtration, there exists isomorphism $\psi:\Ls_1\rightarrow\Ls_2$ making the following diagram
$$\xymatrix{
  F^*_X(\E_1) \ar[r]\ar[d]_{\phi}^{\cong} & \Ls_1 \ar[r]\ar[d]_{\psi}^{\cong} & 0\\
  F^*_X(\E_2) \ar[r] & \Ls_2 \ar[r] & 0}$$
commutative, where the isomorphism $\phi$ is induced from an isomorphism $\E_1\stackrel{\cong}{\rightarrow}\E_2$. By adjunction, we have commutative diagram
$$\xymatrix{
  0 \ar[r] & \E_1 \ar[r]\ar[d]^{\cong} & {F_X}_*(\Ls_1) \ar[d]_{{F_X}_*(\psi)}^{\cong}\\
  0 \ar[r] & \E_2 \ar[r] & {F_X}_*(\Ls_2)}$$
where the horizontal homomorphisms is an isomorphism $${F_X}_*(\psi):{F_X}_*(\Ls_1)\stackrel{\cong}{\rightarrow}{F_X}_*(\Ls_2).$$ This implies $\E_1=\E_2$ as subsheaves of ${F_X}_*(\Ls)$, where $[\Ls]=[\Ls_1]=[\Ls_2]\in\Pic^{(d-1)}(X)$. Thus, $e_1$ and $e_1$ are the some point in the $\Quot^\sharp_X(3,d,\Pic^{(d-1)}(X))$. Hence the morphism $\theta|_{\Quot^\sharp_X(3,d,\Pic^{(d-1)}(X))}$ is injective.

Let $[\E]\in\M^s_X(3,d)$ and $\HNP(F^*_X(\E))\in\{\Ps_2(d),\Ps_3(d)\}$. Then $F^*_X(\E)$ has quotient line bundle $\Ls$ of $\deg(\Ls)=d-1$. Then by Proposition \ref{Prop:Injection}, the adjoint homomorphism $\E\hookrightarrow{F_X}_*(\Ls)$ is an injective homomorphism. Therefore $$e:=[\E\hookrightarrow{F_X}_*(\Ls)]\in\Quot^\sharp_X(3,d,\Pic^{(d-1)}(X))$$ and $\theta(e)=[\E]$. Hence the image of $\theta|_{\Quot^\sharp_X(3,d,\Pic^{(d-1)}(X))}$ is the subset
$$\{~[\E]\in\M^s_X(3,d)~|~\HNP(F^*_X(\E))\in\{\Ps_2(d),\Ps_3(d)\}~\}.$$
\end{proof}

\begin{Theorem}\label{Thm:FrobStra}
Let $k$ be an algebraically closed field of characteristic $3$, and $X$ a smooth projective curve of genus $2$ over $k$. Then $\overline{S_X(3,d,\Ps_i(d))}=S_X(3,d,\Ps^+_i(d))$, and $S_X(3,d,\Ps_i(d))$ $($resp. $S_X(3,d,\Ps^+_i(d))$$)$ are irreducible quasi-projective $($resp. projective$)$ varieties for $1\leq i\leq 4$, $$\dim S_X(3,d,\Ps_i(d))=\dim S_X(3,d,\Ps^+_i(d))=
\begin{cases}
5, i=1\\
5, i=2\\
4, i=3\\
2, i=4\\
\end{cases}$$
\end{Theorem}

\begin{proof}
The morphism
\begin{eqnarray*}
\theta:\Quot_X(3,d,\Pic^{(d-1)}(X))&\rightarrow&\M^s_X(3,d)\\
{[\E\hookrightarrow{F_X}_*(\Ls)]}&\mapsto&[\E]
\end{eqnarray*}
maps $\Quot_X(3,d,\Pic^{(d-1)}(X),\Ps^+_i(d))$ onto $S_X(3,d,\Ps^+_i(d))$ for $2\leq i\leq 4$. Then by Proposition \ref{DimofQuot}, $S_X(3,d,\Ps^+_i(d))$ are irreducible projective varieties and $$\overline{S_X(3,d,\Ps_i(d))}=S_X(3,d,\Ps^+_i(d))$$
for $2\leq i\leq 4$, since $S_X(3,d,\Ps_i(d))$ is an open subvariety of $S_X(3,d,\Ps^+_i(d))$ by Lemma \ref{Shatz77}(2). Moreover, by Proposition \ref{Prop:InjMorphism}, the injection $\theta|_{\Quot^\sharp_X(3,d,\Pic^{(d-1)}(X))}$ maps $\Quot_X(3,d,\Pic^{(d-1)}(X),\Ps_i(d))$ onto $S_X(3,d,\Ps_i(d))$ for $i=2,3$. Then by Proposition \ref{DimofQuot}, we have
$$\dim S_X(3,d,\Ps_i(d))=\dim S_X(3,d,\Ps^+_i(d))=\dim\Quot_X(3,d,\Pic^{(d-1)}(X),\Ps^+_i(d))=
\begin{cases}
5, i=2\\
4, i=3
\end{cases}$$
The isomorphism $\iota:\M^s_X(3,d)\rightarrow\M^s_X(3,-d):[\E]\mapsto[\E^\vee]$ maps $S_X(3,d,\Ps_1(d))$ (resp. $S_X(3,d,\Ps^+_1(d))$) onto $S_X(3,-d,\Ps_2(-d))$ (resp. $S_X(3,-d,\Ps^+_2(-d))$). So we have
$\overline{S_X(3,d,\Ps_1(d))}=S_X(3,d,\Ps^+_1(d))$ is an irreducible projective variety and
$$\dim S_X(3,d,\Ps_1(d))=\dim S_X(3,d,\Ps^+_1(d))=5.$$

Now we study the properties of locus $S_X(3,d,\Ps_4(d))$. By \cite[Lemma 3.1]{Li14}, we know that any $[\E]\in S_X(3,d,\Ps_4(d))$ has the form ${F_X}_*(\Ls')$ for some line bundle $\Ls'$ of $\deg(\Ls')=d-2$. Moreover, by \cite[Theorem 2.5]{Li14}, the morphism
\begin{eqnarray*}
S^s_{\Frob}:\M^s_X(1,d-2)&\rightarrow&\M^s_X(3,d)\\
{[\Ls']}&\mapsto&[{F_X}_*(\Ls')]
\end{eqnarray*}
is a closed immersion and the image of $S^s_{\Frob}$ is just the $S_X(3,d,\Ps_4(d))$. Thus $S_X(3,d,\Ps_4(d))=S_X(3,d,\Ps^+_4(d))$ is isomorphic to Jacobian variety $\mathrm{Jac}_X$ of $X$ which is a smooth irreducible projective variety of dimension $2$.
\end{proof}

The subset of Frobenius destabilised stable vector bundles in $\M^s_X(3,d)$ is a reducible closed subset which contains two irreducible closed subvarieties of dimension $5$. Since $\dim\M^s_X(3,d)=10$, the codimension of the subset of Frobenius destabilised stable vector bundles in $\M^s_X(3,d)$ is $5$.

\section*{Acknowledgements}
I would like to thank Luc Illusie, Xiaotao Sun, Yuichiro Hoshi, Mingshuo Zhou, Junchao Shentu, Yingming Zhang for their interest and helpful conversations. Especially, I express my greatest appreciation to the late Professor Michel Raynaud, who tell me the Lemma \ref{Lem:Grothendieck} according to an unpublished note of Alexander Grothendieck.

\end{document}